\documentclass[12pt]{amsart}
\usepackage{times,amssymb}
\usepackage{hyperref}
\usepackage{graphicx}
\usepackage{comment}

\parskip=\smallskipamount

\newtheorem*{Thm*}{Theorem}
\newtheorem{Thm}{Theorem}
\newtheorem{Cor}[Thm]{Corollary}
\newtheorem{Prop}[Thm]{Proposition}
\newtheorem{Lemma}[Thm]{Lemma}
\newtheorem*{Lemma*}{Lemma}

\theoremstyle{definition}
\newtheorem*{Defn*}{Definition}
\newtheorem{Defn}[Thm]{Definition}

\newtheorem{Remark}[Thm]{Remark}
\newtheorem*{Remark*}{Remark}
\newtheorem{Example}[Thm]{Example}
\newtheorem*{Example*}{Example}

\newcommand{\abs}[1]{\left\vert#1\right\vert}
\newcommand{\eps}{\varepsilon}
\newcommand{\set}[1]{\left\{#1\right\}}
\newcommand{\mf}[1]{\mathbb{#1}}
\newcommand{\mc}[1]{\mathcal{#1}}
\newcommand{\mb}[1]{\mathbf{#1}}
\newcommand{\norm}[1]{\left \|{#1} \right \|}

\DeclareMathOperator{\dsw}{\rho_{\mathit{swap}}}
\DeclareMathOperator{\wst}{\overline{\mathrm{w}}}

\begin{document}

\title[Concentration of the Product of Exponentials]{Product of exponentials concentrates \\ around the exponential of the sum}
\author{Michael Anshelevich, Austin Pritchett}
\address{Department of Mathematics, Texas A\&M University, College Station, TX 77843-3368}
\email{manshel@math.tamu.edu, austinpritchett00@gmail.com}
\thanks{This work was supported in part by a Simons Foundation Collaboration Grant.}

\maketitle

\begin{abstract}
For two matrices $A$ and $B$, and large $n$, we show that most products of $n$ factors of $e^{A/n}$ and $n$ factors of $e^{B/n}$ are close to $e^{A + B}$. This extends the Lie-Trotter formula. The elementary proof is based on the relation between words and lattice paths, asymptotics of binomial coefficients, and matrix inequalities. The result holds for more than two matrices.
\end{abstract}

\section{Introduction.}

Matrix products do not commute. One familiar consequence is that in general,
\[
e^A e^B \neq e^{A + B}.
\]
(Here for a square matrix $A$, the expression $e^A$ can be defined, for example, using the power series expansion of the exponential function.) However, a vestige of the ``product of exponentials is the exponential of the sum'' property remains, as long as we take the factors in a very special \emph{alternating} order.

\begin{Thm*}[Lie-Trotter product formula]
Let $A$ and $B$ be complex square matrices. Then
\[
\lim_{n \rightarrow \infty} \left( e^{A/n} e^{B/n} \right)^n \rightarrow e^{A + B},
\]
where the convergence is with respect to any matrix norm.
\end{Thm*}

This result goes back to Sophus Lie, see \cite{Herzog-Lie} or Proposition~\ref{Prop:one-swap}(b) below for an elementary proof. Clearly, if we take $n$ factors $e^{A/n}$ and $n$ factors $e^{B/n}$ but multiply them in a different order, the result will not always converge to $e^{A+B}$. For example,
\[
\left(e^{A/n} \right)^n \left( e^{B/n} \right)^n = e^A e^B, \quad \left(e^{B/n} \right)^n \left( e^{A/n} \right)^n = e^B e^A.
\]
The reader is invited to try out plotting all of such products for their preferred choices of (real) matrices at \url{https://austinpritchett.shinyapps.io/nexpm_visualization/}

Nevertheless, in this article we show that, for large $n$, the overwhelming majority of products of $n$ factors $e^{A/n}$ and $n$ factors $e^{B/n}$ will be close to $a^{A + B}$. In other words, such products \emph{concentrate} around $e^{A + B}$. To give a precise formulation, we introduce some notation.

\begin{Defn}
Denote by $\mc{W}_n$ the set of all words in $A$ and $B$ which contain exactly $n$ $A$'s and $n$ $B$'s. Denote by $w[i]$ the $i$'th letter in $w$.
\end{Defn}

\begin{Thm}
\label{Thm:Main}
Let $A$ and $B$ be complex square matrices. Consider all $\binom{2n}{n}$ products of $e^{A/n}$ and $e^{B/n}$ of the form $\prod_{i=1}^{2n} e^{w[i]/n}$ for $w \in \mc{W}_n$. Among these products, the proportion of those which differ from $e^{A + B}$ in norm by less than
\[
\sqrt{\frac{\ln n}{n}}
\]
goes to $1$ as $n \rightarrow \infty$.
\end{Thm}

Along the way to the proof of this result, we discuss several metrics on the space of words, which are interesting in their own right.

\textsl{This expanded version also contains an appendix, which does not appear in the published version. In it, we provide several figures illustrating possible shapes of the set of products.}

\section{Words and paths.}

We define three metrics on the set of words $\mc{W}_n$.

\begin{Defn}
Let $w$ be a word. A \emph{swap} is an interchange of two neighboring letters in $w$. The \emph{swap distance} $\dsw(w, v)$ between two words $w, v \in \mc{W}_n$ is the minimal number of swaps needed to transform $w$ into $v$.
\end{Defn}

This metric may remind some readers of the bubble-sort algorithm.

\begin{Example}
We may swap
\[
A A B B \mapsto A B A B \mapsto A B B A \mapsto B A B A \mapsto B B A A.
\]
It is not hard to check that this is the minimal number of swaps needed, so
\[
\dsw(A A B B, B B A A) = 4.
\]
\end{Example}

To define the other two metrics, it is convenient to represent a word by a \emph{lattice path}. To be able to consider words of different length on equal footing, our lattice paths will be normalized.

\begin{Defn}
A lattice path connects the origin $(0,0)$ to the point $(1, 1)$ by a path consisting of $n$ horizontal and $n$ vertical segments of length $1/n$.
\end{Defn}

We may identify words in $\mc{W}_n$ with such paths. For a word $w$, denote
\[
w_A[j] = \#\set{i \leq j : w[i] = A}
\]
the number of $A$'s among the first $j$ letters, and the same for $w_B[j]$. Then the path corresponding to $w$ consists of points
\[
\set{\frac{1}{n} (w_A[j], w_B[j]) : 1 \leq j \leq 2n}
\]
connected by straight line segments, with $A$ corresponding to a horizontal step, and $B$ to a vertical step. See Figure~\ref{Figure:Path}.

\begin{figure}[h]
\label{Figure:Path}
\includegraphics[width=2in]{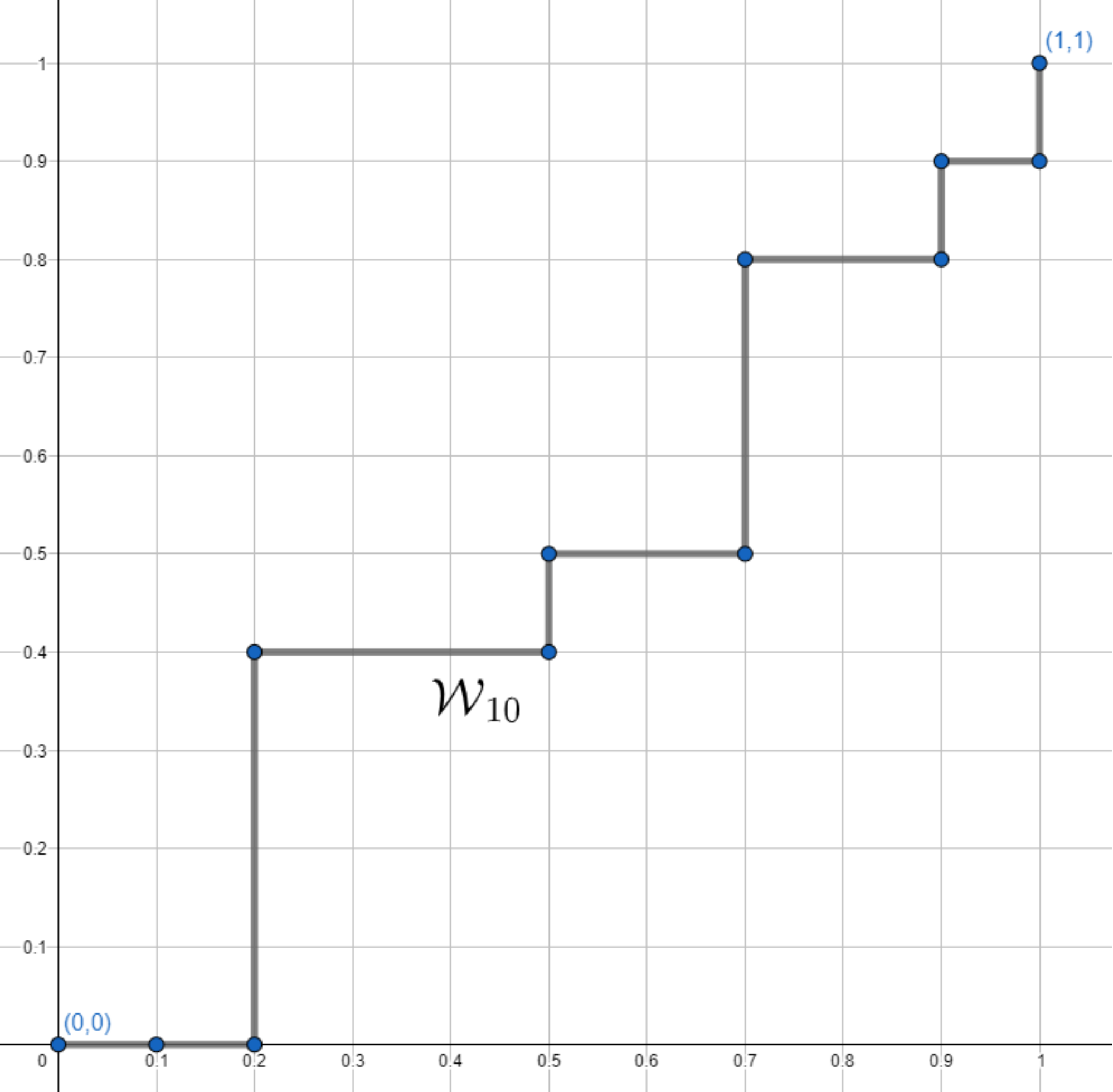}
\caption{The path corresponding to the word $AABBBBAAABAABBBAABAB$.}
\end{figure}

\begin{Defn}
For two words $w, v \in \mc{W}_n$, define the distance $\rho_1(w, v)$ to be the (unsigned) area of the region between the paths. See Figure~\ref{Figure:d1}.
\end{Defn}

\begin{figure}[h]
\label{Figure:d1}
\includegraphics[width=2in]{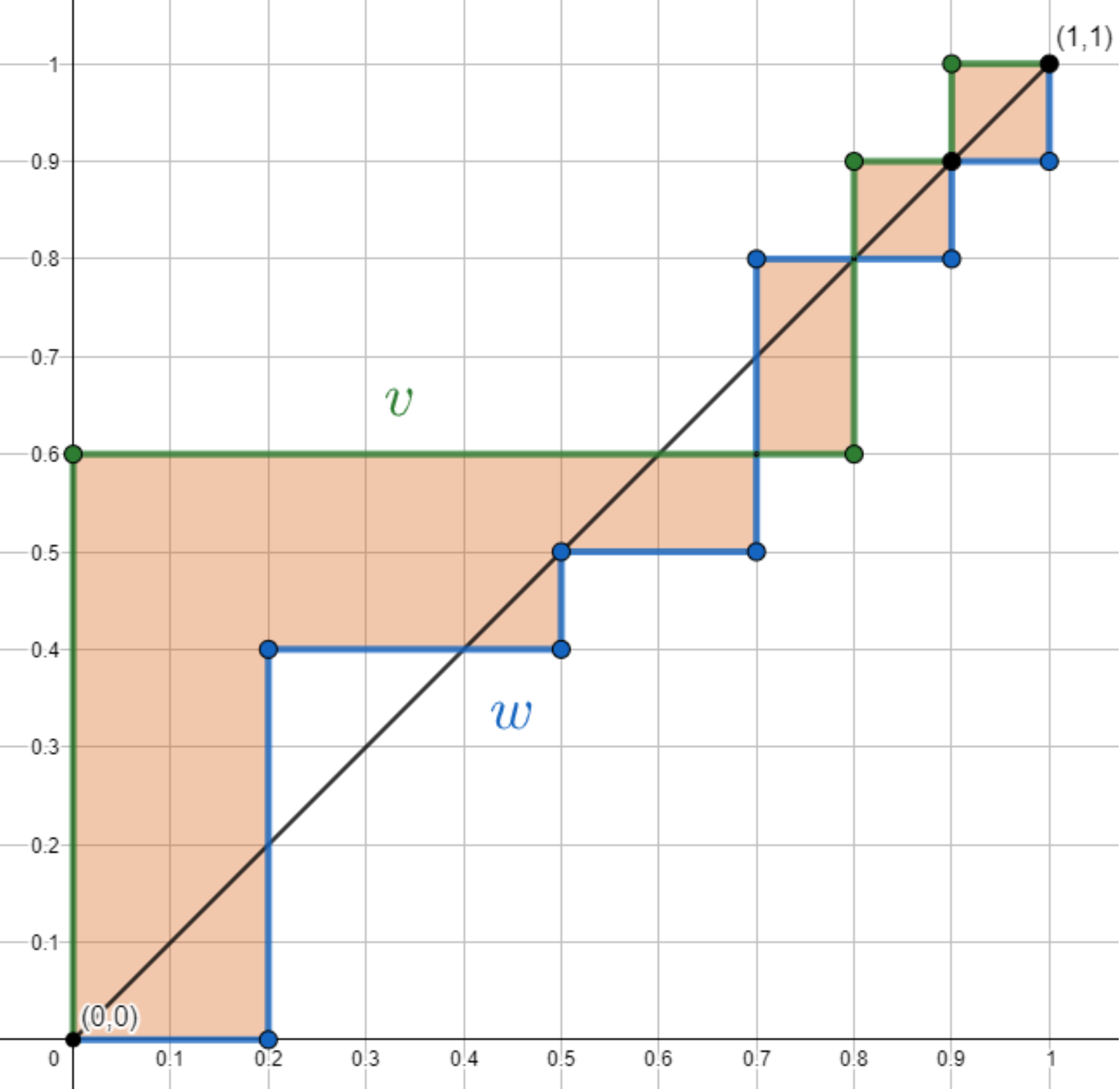}
$\qquad$
\includegraphics[width=2in]{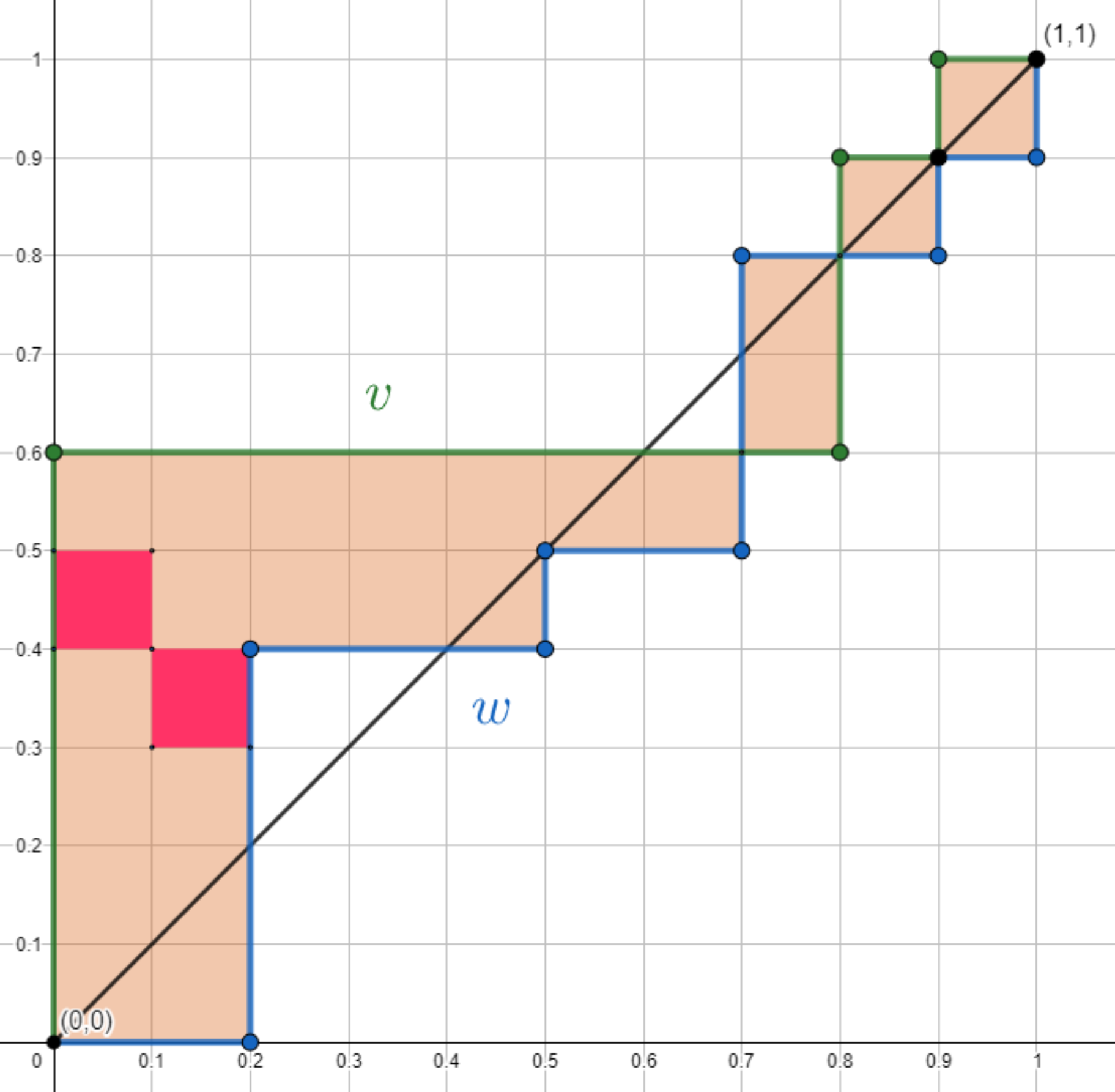}
\caption{$\rho_1(w, v) = \frac{24}{100}$, while $\rho_\infty(w, v) = \frac{4}{10}$. In the second plot, $\abs{w_A[5] - v_A[5]} = 2$.}
\end{figure}

\begin{Lemma}
We can express $\rho_1(w, v)$ directly in terms of the words $w, v$ as follows:
\[
\begin{split}
\rho_1\left(w, v \right)
& = \frac{1}{2 n^2} \sum_{j=1}^{2n} \Bigl|\Bigl(w_A[j] - w_B[j] \Bigr) - \Bigl(v_A[j] - v_B[j] \Bigr) \Bigr| \\
& = \frac{1}{n^2} \sum_{j=1}^{2n} \Bigl| w_A[j] - v_A[j] \Bigr| \\
& = \frac{1}{2 n^2} \sum_{j=1}^{2n} \left( \Bigl|w_A[j] - v_A[j]\Bigr| + \Bigl|w_B[j] - v_B[j]\Bigr| \right).
\end{split}
\]
Here the first representation compares the excess of the number of $A$'s over the number of $B$'s in $w$ and $v$.
\end{Lemma}

\begin{proof}
To obtain the second expression, we slice the region between the paths into NW-SE diagonal regions. For each $j$,
\begin{equation}
\label{Eq:diagonal}
w_A[j] + w_B[j] = v_A[j] + v_B[j] = j,
\end{equation}
and there are exactly $\Bigl| w_A[j] - v_A[j] \Bigr|$ squares located on the diagonal between the points $\frac{1}{n} (w_A[j], w_B[j])$ and $\frac{1}{n} (v_A[j], v_B[j])$. See Figure~\ref{Figure:d1}. For the first and third expressions, again using the identity \eqref{Eq:diagonal},
\[
\begin{split}
\Bigl|\Bigl(w_A[j] - w_B[j] \Bigr) - \Bigl(v_A[j] - v_B[j] \Bigr) \Bigr|
& = 2 \Bigl| w_A[j] - v_A[j] \Bigr| \\
& = \Bigl|w_A[j] - v_A[j]\Bigr| + \Bigl|w_B[j] - v_B[j]\Bigr|. \qedhere
\end{split}
\]
\end{proof}

\begin{Defn}
The third distance we will consider is
\[
\begin{split}
\rho_\infty\left(w, v \right)
& = \frac{1}{n} \max_{1 \leq j \leq 2n} \Bigl|\Bigl(w_A[j] - w_B[j] \Bigr) - \Bigl(v_A[j] - v_B[j] \Bigr) \Bigr| \\
& = \frac{2}{n} \max_{1 \leq j \leq n} \Bigl| w_A[j] - v_A[j] \Bigr| \\
& = \frac{1}{n} \max_{1 \leq j \leq 2n} \left( \Bigl|w_A[j] - v_A[j]\Bigr| + \Bigl|w_B[j] - v_B[j]\Bigr| \right).
\end{split}
\]
It can interpreted as the maximal difference between the corresponding points on the paths as measured in the NW-SE direction, with appropriate normalization.
\end{Defn}

Clearly
\begin{equation}
\label{Eq:Compare-1-infty}
\rho_1 \leq \rho_\infty.
\end{equation}

We now observe that the swap metric and the path metric are related in a simple way.

\begin{Thm}
\label{Thm:metrics-equal}
$\rho_1(w, v) = \dfrac{1}{n^2} \dsw(w, v)$.
\end{Thm}

\begin{proof}
Each swap of neighboring letters changes the area between the paths by $\dfrac{1}{n^2}$. So $\rho_1(w,v) \leq \dfrac{1}{n^2} \dsw(w,v)$. On the other hand, unless the words are equal, we can find an $A$ followed by a $B$ such that at that point in the word, one word has more $A$'s than the other one. Then swapping these $A$ and $B$ decreases $\rho_1$. So one can always transform a word $w$ into $v$ by $n^2 \rho_1(w, v)$ swaps.
\end{proof}

\begin{Prop}
\label{Prop:Refletion}
Let $M \in \set{1, \ldots, n}$. The number of words $w \in \mc{W}_n$ for which the $\rho_\infty$ distance to the \emph{standard word} $\wst = A B A B \ldots A B \in \mc{W}_n$ is at least $M/n$ is at most $2 \binom{2n}{n - M + 1}$.
\end{Prop}

\begin{proof}
Note first that
\[
\begin{split}
\rho_\infty(w, \wst)
& = \frac{1}{n} \max_{1 \leq j \leq 2n} \Bigl|\Bigl(w_A[j] - w_B[j] \Bigr) - \Bigl(\wst_A[j] - \wst_B[j] \Bigr) \Bigr| \\
& \leq \frac{1}{n} \max_{1 \leq j \leq 2n} \Bigl| w_A[j] - w_B[j] \Bigr| + \frac{1}{n}.
\end{split}
\]
Suppose $\rho_\infty(w, \wst) \geq M/n$, so that
\[
\max_{1 \leq j \leq 2n} \Bigl| w_A[j] - w_B[j] \Bigr| \geq M - 1.
\]
We now apply the Reflection Principle, see \cite{Renault-Reflection} or Section 10.3 of \cite{Krattenthaler-Handbook}. Let $j$ be the smallest index such that
\begin{equation}
\label{Eq:Reflection}
\Bigl| w_A[j] - w_B[j] \Bigr| = M - 1;
\end{equation}
note that such a $j$ exists. Let $\tilde{w}$ be a word of length $2n$ (not necessarily in $\mc{W}_n$) such that
\begin{itemize}
\item
for $i \leq j$, $\tilde{w}[i] = w[i]$,
\item
for $i > j$, $\tilde{w}[i] = A$ if $w[i] = B$, and $\tilde{w}[i] = B$ if $w[i] = A$.
\end{itemize}
Since $w_A[j] = w_B[j] \pm (M - 1)$,  $\tilde{w}$ contains
\[
w_A[j] + (n - w_B[j])
= n \pm (M - 1)
\]
$A$'s. Conversely, because $j$ is the \emph{smallest} index satisfying equation~\eqref{Eq:Reflection}, this procedure can be reversed, and each word $v$ with $n \pm (M - 1)$ $A$'s arises as $\tilde{w}$ for a unique $w \in \mc{W}_n$. It remains to note that the total number of words of length $2n$ with $n \pm (M - 1)$ $A$'s is
\[
2 \binom{2n}{n - M + 1}. \qedhere
\]
\end{proof}

\begin{Remark}
Recall the little-o, big-O, and asymptotic notation. For two positive sequences $(a_n)$ and $(b_n)$, we write
\begin{itemize}
\item
$a_n = o(b_n)$ if $a_n/b_n \rightarrow 0$
\item
$a_n = O(b_n)$ is $a_n/b_n$ is bounded
\item
$a_n \sim b_n$ if $a_n/b_n \rightarrow 1$
\end{itemize}
\end{Remark}

\begin{Cor}
\label{Cor:Stirling}
Let $p(n)$ be a positive sequence such that $p(n) \rightarrow \infty$ and $p(n) = o(n^{1/6})$. Then for large $n$, the proportion of words $w \in \mc{W}_n$ for which
\[
\rho_\infty(w, \wst) \geq \frac{p(n)}{\sqrt{n}}
\]
is asymptotically at most $2 e^{-p(n)^2}$.
\end{Cor}

\begin{proof}
The desired proportion is at most
\[
2 \frac{\binom{2n}{n - [\sqrt{n} p(n)] + 1}}{\binom{2n}{n}},
\]
where $[\sqrt{n} p(n)]$ denotes the integer part. The asymptotics of this expression can be found using Stirling's formula, see equation (5.43) in \cite{Spencer-Asymptotia}.
\end{proof}

\begin{Remark}
$\rho_\infty$ is closely related to the notion of ``span'' from \cite{Prodinger-expected-height}, where its asymptotic expected value is computed (more precisely, ``span'' is the $\tau$ statistic in the final section below). Related analysis for paths which lie entirely above the main diagonal is sometimes called the Sock Counting Problem, for reasons we invite the reader to discover.
\end{Remark}

\section{Matrix estimates.}

We now recall that $A, B$ are actually matrices in $M_d(\mf{C}) = \mf{C}^{d \times d}$. Denote by $\norm{\cdot}$ some sub-multiplicative norm on this matrix space, such as the operator norm or the Frobenius norm.

\begin{Lemma}
\label{Lemma:Bounded}
For every word $w \in \mc{W}_n$
\[
\norm{e^{w[1]/n} e^{w[2]/n} \cdots e^{w[2n]/n}} \leq e^{\norm{A} + \norm{B}},
\]
with a uniform bound which does not depend on the word or on $n$.
\end{Lemma}

\begin{proof}
Since the norm is sub-multiplicative,
\[
\norm{e^C} = \norm{\sum_{k=0}^\infty \frac{1}{k!} C^k} \leq \sum_{k=0}^\infty \frac{1}{k!} \norm{C}^k = e^{\norm{C}}.
\]
Therefore
\[
\begin{split}
\norm{e^{w[1]/n} e^{w[2]/n} \cdots e^{w[2n]/n}}
& \leq \prod_{i=1}^{2n} \norm{e^{w[i]/n}} \leq \prod_{i=1}^{2n} e^{\norm{w[i]}/n} \\
& = e^{\sum_{i=1}^{2n} \norm{w[i]}/n} = e^{\norm{A} + \norm{B}}. \qedhere
\end{split}
\]
\end{proof}

The following estimates can be improved (with a longer argument), but suffice for our purposes.

\begin{Lemma}
\label{Lemma:Product-two}
For large $n$,
\[
\norm{e^{A/n} e^{B/n} - e^{(A + B)/n}}
\leq \frac{1}{n^2} \norm{A B - B A}
\]
and
\[
\norm{e^{A/n} e^{B/n} - e^{B/n} e^{A/n}} \leq \frac{2}{n^2} \norm{A B - B A}.
\]
\end{Lemma}

\begin{proof}
If $A B = B A$, the result is immediate, so we assume that $A B \neq B A$. Then
\[
\begin{split}
& \norm{e^{A/n} e^{B/n} - e^{(A + B)/n} - \frac{1}{2 n^2} (A B - B A)} \\
&\quad = \norm{\sum_{k=0}^\infty \frac{1}{n^k} \sum_{\ell=0}^k \frac{1}{\ell! (k - \ell)!} A^\ell B^{k - \ell} - \sum_{k=0}^\infty \frac{1}{n^k} \frac{1}{k!} (A + B)^k  - \frac{1}{2 n^2} (A B - B A)} \\
&\quad = \norm{\sum_{k=3}^\infty \frac{1}{n^k} \frac{1}{k!} \left( \sum_{\ell=0}^k \binom{k}{\ell} A^\ell B^{k - \ell} - (A + B)^k \right)} \\
&\quad \leq 2 \sum_{k=3}^\infty \frac{1}{n^k} \frac{1}{k!} (\norm{A} + \norm{B})^k \\
&\quad \leq \frac{2}{n^3} e^{\norm{A} + \norm{B}}.
\end{split}
\]
Therefore
\[
\norm{e^{A/n} e^{B/n} - e^{(A + B)/n}}
\leq \frac{1}{2 n^2} \norm{A B - B A } + \frac{2}{n^3} e^{\norm{A} + \norm{B}}
\leq \frac{1}{n^2} \norm{A B - B A}
\]
for large $n$. The second estimate follows from the first.
\end{proof}

\begin{Prop}
\label{Prop:one-swap}
\

\begin{enumerate}
\item
Swapping two neighboring letters in a word changes the corresponding product by $O(1/n^2)$. More precisely,
\begin{multline*}
\Bigl\| e^{w[1]/n} \cdots e^{w[i]/n} e^{w[i+1]/n} \cdots e^{w[2n]/n} \\
- e^{w[1]/n} \cdots e^{w[i+1]/n} e^{w[i]/n} \cdots e^{w[2n]/n} \Bigr\|
\leq \frac{2}{n^2} \norm{A B - B A}  e^{\norm{A} + \norm{B}}.
\end{multline*}
\item
The Lie-Trotter formula:
\[
\norm{\left(e^{A/n} e^{B/n} \right)^n - e^{A + B}}
\leq \frac{1}{n} \norm{A B - B A} e^{\norm{A} + \norm{B}}.
\]
\end{enumerate}
\end{Prop}

\begin{proof}
For (a), using the two preceding lemmas,
\[
\begin{split}
& \Bigl\| e^{w[1]/n} \cdots e^{w[i]/n} e^{w[i+1]/n} \cdots e^{w[2n]/n} \\
&\qquad - e^{w[1]/n} \cdots e^{w[i+1]/n} e^{w[i]/n} \cdots e^{w[2n]/n} \Bigr\| \\
&\quad = \Bigl\| e^{w[1]/n} \cdots e^{w[i-1]/n} \\
&\quad\qquad \times \Bigl(e^{w[i]/n} e^{w[i+1]/n} - e^{w[i+1]/n} e^{w[i]/n} \Bigr) e^{w[i+2]/n} \cdots e^{w[2n]/n} \Bigr\| \\
&\quad \leq e^{\norm{A} + \norm{B}} \norm{e^{A/n} e^{B/n} - e^{B/n} e^{A/n}} \\
&\quad \leq \frac{2}{n^2} e^{\norm{A} + \norm{B}} \norm{A B - B A}.
\end{split}
\]
Similarly, for (b),
\[
\begin{split}
\norm{\left(e^{A/n} e^{B/n} \right)^n - e^{A + B}}
& = \norm{\left(e^{A/n} e^{B/n} \right)^n - \left(e^{(A + B)/n}\right)^n} \\
& \leq n \norm{e^{A/n} e^{B/n} - e^{(A + B)/n}} e^{\norm{A} + \norm{B}} \\
& \leq \frac{1}{n} \norm{A B - B A} e^{\norm{A} + \norm{B}}. \qedhere
\end{split}
\]
\end{proof}

\begin{Thm}
\label{Thm:Continuity}
For fixed matrices $A, B$, the map $F: (\mc{W}_n, \rho_1) \rightarrow (M_d(\mf{C}), \norm{\cdot})$ given by
\[
F(w) = e^{w[1]/n} e^{w[2]/n} \cdots e^{w[2n]/n}
\]
is Lipschitz continuous, with the Lipschitz constant independent of $n$.
\end{Thm}

\begin{proof}
By Proposition~\ref{Prop:one-swap}(a) and Theorem~\ref{Thm:metrics-equal}, for any two words $w, v \in \mc{W}_n$,
\begin{equation}
\label{Eq:Distance-matrices}
\begin{split}
\norm{\prod_{i=1}^{2n} e^{w[i]/n} - \prod_{i=1}^{2n} e^{v[i]/n}}
& \leq \frac{\dsw(w, v)}{n^2} 2 \norm{A B - B A}  e^{\norm{A} + \norm{B}} \\
& = \rho_1(w, v) 2 \norm{A B - B A}  e^{\norm{A} + \norm{B}}.
\end{split}
\end{equation}
So the map $F$ is Lipschitz continuous, with the constant depending only on $A$ and $B$.
\end{proof}

\begin{Remark}
\label{Remark:Increasing}
One can identify the lattice paths discussed above with non-decreasing step functions from $[0,1]$ to $[0,1]$ which (for some $n$) take values in $\set{\frac{k}{n} : 0 \leq k \leq n}$ and are constant on the intervals in the uniform partition of $[0,1]$ into $n$ subintervals. It is easy to check that the closure of this space, with respect to the metric $\rho_1$, is the space of \emph{all} increasing functions from $[0,1]$ to $[0,1]$.
By Theorem~\ref{Thm:Continuity}, the map $F$ extends continuously to a map from the space of all such increasing functions (with the $\rho_1$ metric) to $M_d(\mf{C})$.
\end{Remark}

\section{The main result.}

\begin{proof}[Proof of Theorem~\ref{Thm:Main}]
Fix $c > 0$. Applying Corollary~\ref{Cor:Stirling} with $p(n) = \sqrt{c \ln n}$, the proportion of words $w \in \mc{W}_n$ for which
\[
\rho_\infty(w, \wst) \geq \sqrt{c \frac{\ln n}{n}}
\]
is asymptotically at most
\[
2 e^{- c \ln n} = \frac{2}{n^{c}}.
\]
On the other hand, by inequality \eqref{Eq:Compare-1-infty}, for $w$ with
\[
\rho_\infty(w, \wst) < \sqrt{c \frac{\ln n}{n}}
\]
we also have
\[
\rho_1(w, \wst) < \sqrt{c \frac{\ln n}{n}}.
\]
By equation~\eqref{Eq:Distance-matrices}, for such $w$,
\[
\norm{F(w) - F(\wst)} \leq \sqrt{c \frac{\ln n}{n}} 2 \norm{A B - B A}  e^{\norm{A} + \norm{B}}.
\]
Finally, by Proposition~\ref{Prop:one-swap}(b), for such $w$,
\begin{equation}
\label{Eq:Distance-Exp}
\begin{split}
\norm{F(w) - e^{A + B}}
& \leq \left( \sqrt{c \frac{\ln n}{n}} + \frac{1}{2n} \right) 2 \norm{A B - B A} e^{\norm{A} + \norm{B}} \\
& \leq 4 \sqrt{c \frac{\ln n}{n}} \norm{A B - B A} e^{\norm{A} + \norm{B}}
\end{split}
\end{equation}
for large $n$.

If $A B = B A$, then $F(w) = e^{A + B}$ for each $w \in \mc{W}_n$. If $A B \neq B A$, set
\[
c = \frac{1}{\left(4 \norm{A B - B A} e^{\norm{A} + \norm{B}}\right)^2}.
\]
It follows that the proportion of words $w \in \mc{W}_n$ with $\norm{F(w) - e^{A + B}} < \sqrt{\frac{\ln n}{n}}$ is at least
\[
1 - \frac{2}{n^{c}},
\]
and so goes to one as $n \rightarrow \infty$.
\end{proof}

For readers familiar with probability theory, we can state a cleaner result. We will need the following device.

\begin{Lemma*}[Borel-Cantelli]
Denote by $P(E)$ the probability of an event. If a family of events $\set{E_n : n \in \mf{N}}$ has the property that the series $\sum_{n=1}^\infty P(E_n) < \infty$, then almost surely, an element $x$ lies in at most finitely many $E_n$'s.
\end{Lemma*}

\begin{Cor}
Let $\mc{W}_n$ and $F$ be as before. Put on $\mc{W}_n$ the uniform measure, so that each word has probability $\dfrac{1}{\binom{2n}{n}}$. Let $\mc{W} = \prod_{n=1}^\infty \mc{W}_n$ be the collection of all sequences of words of progressively longer length, with the usual product probability measure. Then for $\mb{w} = (w_1, w_2, \ldots) \in \mc{W}$, almost surely with respect to this product measure,
\[
F(w_n) \rightarrow e^{A + B}
\]
in the matrix norm as $n \rightarrow \infty$.
\end{Cor}

\begin{proof}
In the proof of Theorem~\ref{Thm:Main}, take $c > 1$. Since the series $\sum \frac{1}{n^{c}}$ converges, combining equation~\eqref{Eq:Distance-Exp} and the Borel-Cantelli lemma, almost surely
\[
\norm{F(w_n) - e^{A + B}} \leq 4 \sqrt{c \frac{\ln n}{n}} \norm{A B - B A} e^{\norm{A} + \norm{B}}
\]
for all but finitely many $n$. This implies that $F(w_n) \rightarrow e^{A + B}$.
\end{proof}

\begin{Remark}
We don't need the full power of Theorem~\ref{Thm:Main} to prove the preceding corollary. Indeed, all we need is that for any $\eps > 0$, $\norm{F(w_n) - e^{A + B}} < \eps$ for all but finitely many terms. This corresponds to the asymptotics of the binomial coefficient $\binom{2n}{n - [n \eps]}$ for $\eps< 1$. These asymptotics (describing the large rather than moderate deviations from the mean) are both easier and better known, namely
\[
\binom{2n}{n - [n \eps]} \sim \sqrt{1 - \eps^2} e^{-H(\eps) n} \binom{2n}{n},
\]
see for example Section 5.3 in \cite{Spencer-Asymptotia}. Here
\[
H(\eps) = (1 + \eps) \ln (1 + \eps) + (1-\eps) \ln (1-\eps).
\]
Since the function $x \ln x$ is concave up, $H(\eps) \geq 0$. So the series $\sum_n e^{-H(\eps) n}$ converges, and the Borel-Cantelli lemma still implies the result.
\end{Remark}

\section{The case of several matrices.}

Similar results hold if instead of $A$ and $B$, we start with an $N$-tuple of matrices $A_1, \ldots, A_N \in M_d(\mf{C})$. Several parts of the argument require modification, while others are almost the same (and so are only outlined).

\begin{Defn}
Let $\mc{W}_n^{(N)}$ be the collection of all words of length $N n$ containing exactly $n$ of each $A_j$, $1 \leq j \leq N$. Define the standard word $\wst$ to be the word $A_1 A_2 \ldots A_N$ repeated $n$ times. Define the swap distance exactly as before,
\[
\rho_1\left(w, v \right)
= \frac{1}{N^2 n^2} \sum_{j=1}^{N n} \sum_{k=1}^N \Bigl| w_{A_k}[j] - v_{A_k}[j] \Bigr|,
\]
and
\[
\rho_\infty\left(w, v \right)
= \frac{2}{n} \max_{\substack{1 \leq k \leq N \\ 1 \leq j \leq N n}} \Bigl| w_{A_k}[j] - v_{A_k}[j] \Bigr|.
\]
We also define $F$ by exactly the same formula as in Theorem~\ref{Thm:Continuity}.
\end{Defn}

\begin{Example}
$\dsw$ and $\rho_1$ no longer determine each other. Indeed, omitting the normalization factor,
\[
\rho_1(A C B, B C A) =  2 + 2 + 0 = 4 \quad \text{while} \quad \dsw(A C B, B C A) = 3.
\]
On the other hand,
\[
\rho_1(A B C, B C A) = 2 + 1 + 1 = 4 \quad \text{while} \quad \dsw(A B C, B C A) = 2.
\]
\end{Example}

However, all we really need is an inequality between $\dsw$ and $\rho_\infty$, which still holds.

\begin{Prop}
$\dsw \leq \frac{1}{2} N^2 n^2 \rho_\infty$.
\end{Prop}

\begin{proof}
Suppose the $i$'th position is the first one where $w$ and $v$ differ, and $w[i] = A_k$. Then the next $A_k$ appears in $v$ no later than $N \frac{n}{2} \rho_\infty(w,v)$ positions away. So no more than $(N n) \cdot N \frac{n}{2}  \rho_\infty(w,v)$ swaps are necessary to transform $v$ into $w$.
\end{proof}

Counting exactly the number of words which lie within a given $\rho_\infty$ distance from the standard word is a difficult question, see Section 10.17 in \cite{Krattenthaler-Handbook}. For our needs, the following slightly different estimate suffices.

\begin{Defn}
Let $w \in W_n^{(N)}$. Denote
\[
\tau(w) = \frac{1}{n} \max_{\substack{1 \leq k, \ell \leq N \\ 1 \leq j \leq N n}} \Bigl|w_{A_k}[j] - w_{A_\ell}[j]\Bigr|.
\]
\end{Defn}

Just like $\rho_\infty(w, \wst)$, $\tau(w)$ measures how far the path corresponding to $w$ is from the straight path connecting the origin to $(n, n, \ldots, n)$. In fact,

\begin{Lemma}
For any $w \in W_n^{(N)}$,
\[
\frac{1}{2} \rho_\infty(w, \wst) \leq \tau(w) + \frac{1}{n}.
\]
\end{Lemma}

\begin{proof}
Note that $\wst_{A_k}[j] = \left[\frac{j + N - k}{N} \right]$ and $\sum_{\ell=1}^N w_{A_\ell}[j] = j$. So
\[
\begin{split}
\Bigl|w_{A_k}[j] - \wst_{A_k}[j]\Bigr|
& \leq \abs{w_{A_k}[j] - \frac{j}{N}} + \abs{\left[\frac{j + N - k}{N} \right] - \frac{j}{N}} \\
& \leq \frac{1}{N} \sum_{\ell=1}^N \Bigl|w_{A_k}[j] - w_{A_\ell}[j]\Bigr| + 1 \\
& \leq \max_{k, \ell} \Bigl|w_{A_k}[j] - w_{A_\ell}[j]\Bigr| + 1. \qedhere
\end{split}
\]
\end{proof}

\begin{Prop}
\

\begin{enumerate}
\item
The number of words $w \in \mc{W}_n^{(N)}$ for which the $\rho_\infty$ distance to the standard word is at least $(2M + 2)/n$ is at most $2 N^2 \binom{N n}{n - M, n + M, n, \ldots, n}$.
\item
Let $p(n) = o(n^{1/6})$. Then the proportion of words $w \in \mc{W}_n$ for which the $\rho_\infty$ distance to the standard word is at least $\frac{p(n)}{\sqrt{n}}$ goes to zero as $n \rightarrow \infty$.
\end{enumerate}
\end{Prop}

\begin{proof}
For part (a), by the preceding lemma, it suffices to consider words with $n \tau(w) \geq M$. Let $j$ be the smallest index such that for some $k, \ell$,
\[
\Bigl|w_{A_k}[j] - w_{A_\ell}[j]\Bigr| = M,
\]
and let $k$ and $\ell$ be the indices for which this occurs. Then applying the reflection principle as in Proposition~\ref{Prop:Refletion} just to the letters $A_k$ and $A_\ell$ (keeping all the other letters in their places), the number of such paths is at most the multinomial coefficient $2 \binom{N n}{n - M, n + M, n, \ldots, n}$. Since there are $N^2$ choices for the pair $(k, \ell)$, the result follows.

For part (b), we note that the ratio of multinomial coefficients
\[
\frac{\binom{N n}{n - M, n + M, n, \ldots, n}}{\binom{N n}{n, \ldots, n}} = \frac{n! n!}{(n - M)! (n + M)!} = \frac{\binom{2n}{n - M}}{\binom{2n}{n}}.
\]
So the direct application of Corollary~\ref{Cor:Stirling} gives the result.
\end{proof}

\begin{Cor}
Put on $\mc{W}_n^{(N)}$ the uniform measure, and let $\mc{W}^{(N)} = \prod_{n=1}^\infty \mc{W}_n^{(N)}$, with the usual product probability measure. Then for $\mb{w} = (w_1, w_2, \ldots) \in \mc{W}^{(N)}$, almost surely with respect to this product measure,
\[
F(w_n) \rightarrow e^{A_1 + \ldots + A_N}
\]
in the matrix norm as $n \rightarrow \infty$.
\end{Cor}

\textbf{Acknowledgements.} The first author would like to thank Matthew Junge, who reminded him that words can be treated as random walks. The authors are grateful to Harold Boas for numerous comments, which led to a substantial improvement of the article (the remaining errors are, of course, our own), and to the reviewers for a careful reading of the manuscript and helpful comments.

\appendix

\section{The set of products.}
\label{Sec:Set}

We now return to the case of two matrices. The results earlier in the paper are concerned with the asymptotic \emph{density} of the sets $F(\mc{W}_n)$. A quite different question is to give an asymptotic description of these sets themselves, or perhaps of the closure
\[
\mc{S} = \overline{\bigcup_{n=1}^\infty F(\mc{W}_n)} \subset M_d(\mf{C}).
\]
It is clear that $\mc{S}$ is connected, closed, and bounded. It is also easy to check that all of its elements have same determinant, and so lie in a hypersurface. Beyond these elementary properties, we do not in general have a good description of this set. Several examples are included in Figure~\ref{Figure:Sets}. The two-dimensional images may be hard to read, so the reader is invited to take advantage of the three-dimensional functionality at \url{https://austinpritchett.shinyapps.io/nexpm_visualization/}

\begin{figure}[h]
\label{Figure:Sets}
\includegraphics[width=3in]{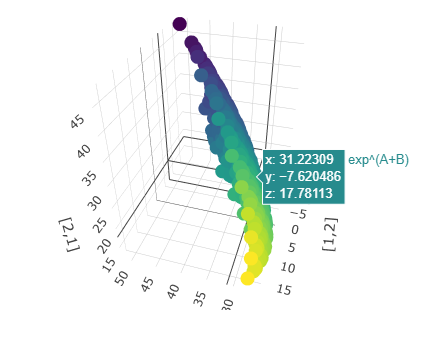}
\includegraphics[width=3in]{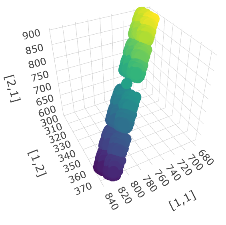}
\includegraphics[width=3in]{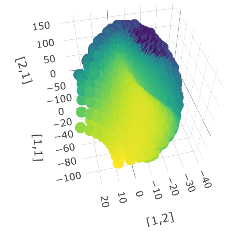}
\includegraphics[width=3in]{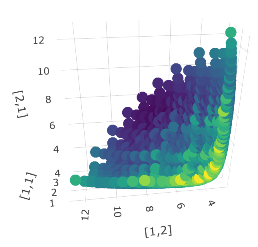}
\caption{The sets $F(\mc{W}_8)$ for several choices of $A$ and $B$. We plot the $(1,1), (1,2)$, and $(2,1)$ entries, and indicate the value of the $(2.2)$ entry via the color.}
\end{figure}

We finish with an example where the set $\mc{S}$, and in fact the function $F : \mc{W}_n \rightarrow M_2(\mf{C})$, can be described completely. Denote by $E_{ij}$ the matrix with a $1$ in the $(i,j)$ position, and $0$'s elsewhere.

\begin{Remark}
Recall that in Remark~\ref{Remark:Increasing} we identified a word with a non-decreasing step function from $[0,1]$ to $[0,1]$. Here is an explicit description of this correspondence. For the $i$'th $A$ in $w$, denote by $h_i(w)$ the number of $B$'s which have appeared in $w$ before it. Equivalently, the $i$'th $A$ appears in position $i + h_i(w)$ in $w$. Then the function $L_w : [0,1] \rightarrow [0,1]$ corresponding to $w$ takes the value $h_i(w)$ on the interval $\left( \frac{i-1}{n}, \frac{i}{n} \right)$. See Figure~\ref{Figure:Path}.
\end{Remark}

\begin{Thm}
Let $A = E_{12}$ and $B = E_{11}$ in $M_2(\mf{C})$. Using the notation just above,
\begin{enumerate}
\item
For $w \in \mc{W}_n$,
\[
F(w) = \begin{pmatrix} e & \frac{1}{n} \left( e^{h_1(w)/n} + \ldots + e^{h_n(w)/n} \right) \\ 0 & 1 \end{pmatrix}.
\]
\item
For a general increasing function $L : [0,1] \rightarrow [0,1]$ as in Remark~\ref{Remark:Increasing},
\[
F(L) = \begin{pmatrix} e & \int_0^1 e^{L(x)} \,dx \\ 0 & 1 \end{pmatrix}.
\]
\end{enumerate}
\end{Thm}

\begin{proof}
It suffices to prove recursively that for any $1 \leq j \leq 2n$,
\[
\prod_{i=1}^j e^{w[i]/n} = \begin{pmatrix} e^{w_B[j]/n} & \frac{1}{n} \sum_{i=1}^{w_A[j]} e^{h_i(w)/n} \\ 0 & 1 \end{pmatrix}.
\]
First consider $j=1$. If the first letter of $w$ is $A$,
\[
e^{A/n} = \begin{pmatrix} 1 & \frac{1}{n} \\ 0 & 1 \end{pmatrix} = \begin{pmatrix} e^{w_B[1]/n} & \frac{1}{n} \sum_{i=1}^{w_A[1]} e^{h_i(w)/n} \\ 0 & 1 \end{pmatrix}
\]
since $h_1(w) = 0$. If the first letter of $w$ is $B$,
\[
e^{B/n} = \begin{pmatrix} e^{1/n} & 0 \\ 0 & 1 \end{pmatrix} = \begin{pmatrix} e^{w_B[1]/n} & \frac{1}{n} \sum_{i=1}^{w_A[1]} e^{h_i(w)/n} \\ 0 & 1 \end{pmatrix}
\]
since the latter sum is empty. Now recursively, if $w[j+1] = A$,
\[
\begin{split}
& \begin{pmatrix} e^{w_B[j]/n} & \frac{1}{n} \sum_{i=1}^{w_A[j]} e^{h_i(w)/n} \\ 0 & 1 \end{pmatrix} \ \begin{pmatrix} 1 & \frac{1}{n} \\ 0 & 1 \end{pmatrix} \\
&\qquad = \begin{pmatrix} e^{w_B[j]/n} & \frac{1}{n} \sum_{i=1}^{w_A[j]} e^{h_i(w)/n} + \frac{1}{n} e^{w_B[j]/n} \\ 0 & 1 \end{pmatrix} \\
&\qquad = \begin{pmatrix} e^{w_B[j+1]/n} & \frac{1}{n} \sum_{i=1}^{w_A[j+1]} e^{h_i(w)/n} \\ 0 & 1 \end{pmatrix}
\end{split}
\]
Indeed, since $w[j+1]$ is the $w_A[j+1]$'th $A$ in $w$, we have $w_B[j+1] = w_B[j] = h_i(w)$ for $i = w_A[j+1]$. Similarly, if $w[j+1] = B$,
\[
\begin{split}
& \begin{pmatrix} e^{w_B[j]/n} & \frac{1}{n} \sum_{i=1}^{w_A[j]} e^{h_i(w)/n} \\ 0 & 1 \end{pmatrix} \ \begin{pmatrix} e^{1/n} & 0 \\ 0 & 1 \end{pmatrix} \\
&\qquad = \begin{pmatrix} e^{(w_B[j] + 1)/n} & \frac{1}{n} \sum_{i=1}^{w_A[j]} e^{h_i(w)/n} \\ 0 & 1 \end{pmatrix} \\
&\qquad = \begin{pmatrix} e^{w_B[j+1]/n} & \frac{1}{n} \sum_{i=1}^{w_A[j+1]} e^{h_i(w)/n} \\ 0 & 1 \end{pmatrix}
\end{split}
\]
since $w_A[j+1] = w_A[j]$.

Part (b) follows: the expression in part (a) is the Riemann sum for the integral $\int_0^1 e^{L(x)} \,dx$, and since the function $L$ is increasing, it is Riemann integrable.
\end{proof}

\begin{Remark}
In the example above,  $\mc{S}$ is a (one-dimensional) curve from $e^A e^B$ to $e^B e^A$. There are several general situations where this behavior also occurs. Denoting $[A, B] = A B - B A$ the commutator of $A$ and $B$, these include
\begin{itemize}
\item
Quasi-commuting matrices \cite{McCoy-Quasi-commutative} for which the commutator is non-zero but commutes with both $A$ and $B$,
\item
Matrices which satisfy $[A, B] = s B$,
\item
\[
A = \begin{pmatrix} x & 1 \\ 0 & x \end{pmatrix}, \quad B = \begin{pmatrix} a & 0 \\ 0 & b \end{pmatrix}
\]
for $a \neq b$.
\end{itemize}
\end{Remark}


\def\cprime{$'$} \def\cprime{$'$}
\providecommand{\bysame}{\leavevmode\hbox to3em{\hrulefill}\thinspace}
\providecommand{\MR}{\relax\ifhmode\unskip\space\fi MR }
\providecommand{\MRhref}[2]{%
  \href{http://www.ams.org/mathscinet-getitem?mr=#1}{#2}
}
\providecommand{\href}[2]{#2}

\end{document}